\documentclass[12pt]{amsart}
\pdfoutput=1
\usepackage{hyperref}
\usepackage{mathpazo,graphicx}
\usepackage{amsmath,amssymb}

\theoremstyle{plain}
\newtheorem{theorem}{Theorem}
\newtheorem*{lemma}{Lemma}
\newtheorem*{longquotation}{Extended Paraphrase}
\newtheorem*{proposition}{Proposition}

\newtheorem*{conjecture}{Generalized Newman Conjecture}

\newcommand{\GRHp}{\text{GRH}^+}

\newcommand{\tim}{\text{Im}}

\newcommand{\lkr}{\Lambda_{\text{Kr}}}

\theoremstyle{definition}
\newtheorem*{remark}{Remark}

\newtheorem*{definition}{Definition}

\theoremstyle{definition}

\begin{document}
\title[Notes on Low discriminants]{Notes on Low discriminants\\and the Generalized Newman Conjecture}
\author{Jeffrey Stopple}
\begin{abstract}
Generalizing work of Polya, de Bruijn and Newman, we allow the backward heat equation to deform the zeros of quadratic Dirichlet $L$-functions.  There is a real constant $\lkr$ (generalizing the de Bruijn-Newman constant $\Lambda$) such that for time $t\ge\lkr$ all such $L$-functions have all their zeros on the critical line; for time $t<\lkr$ there exist zeros off the line.  Under GRH, $\lkr\le 0$; we make the complementary conjecture $0\le \lkr$.  Following the work of Csordas \emph{et.\ al}.\ on Lehmer pairs of Riemann zeros, we use low-lying zeros of quadratic Dirichlet $L$-functions to show that $-1.13\cdot 10^{-7}<\lkr$.  In the last section we develop a precise definition of a Low discriminant which is motivated by considerations of random matrix theory.  The existence of infinitely many Low discriminants would imply $0\le \lkr$.
\end{abstract}
\email{stopple@math.ucsb.edu}\address{Mathematics Department, UC Santa Barbara, Santa Barbara CA 93106}

\maketitle

%\tableofcontents

For $-D<0$ a fundamental discriminant, and $\chi$ the Kronecker symbol, consider the $L$-functions $L(s,\chi)$.  We will assume the Generalized Riemann Hypothesis that the nontrivial zeros of $L(s,\chi)$ are on the critical line.  In fact we'll assume a little more, that also $L(1/2,\chi)\ne0$, which we will denote $\GRHp$.

We define, for  $s=1/2+it$, %\footnote{Alternatively, we could look at the Dedekind zeta function by using \cite{XjL}.}
\begin{align*}
\Xi(t,\chi)\overset{\text{def.}}=&\left(\frac{D}{\pi}\right)^{(s+1)/2}\Gamma((s+1)/2)L(s,\chi),\\
=&\int_0^\infty \Phi(u,\chi)\cos(ut) \, du,
\end{align*}
where 
\begin{equation}\label{Eq:phidef}
\Phi(u,\chi)=4\sum_{n=1}^\infty\chi(n)n\exp(3u/2-n^2\pi\exp(2u)/D).
\end{equation}
Analogous to the Hardy function $Z(t)$ for the Riemann zeta function we have
\[
Z(t,\chi)=\left(\frac{D}{\pi}\right)^{it/2}\left(\frac{\Gamma(3/4+it/2)}{\Gamma(3/4-it/2)}\right)^{1/2}L(1/2+it,\chi),
\]
so that
\[
\Xi(t,\chi)=\left(\frac{D}{\pi}\right)^{3/4}\left|\Gamma(3/4+it/2)\right|Z(t,\chi).
\]

\subsection*{Low discriminants}
In 1965, Marc Low wrote a Ph.D.\ thesis \cite{Low} investigating possible real zeros of $L(s,\chi)$.  He was able to prove that
$L(s,\chi)$ has no real zeros for $-593000<-D$, with the possible exception of $-D=-115147$.
Imagine his frustration at being unable to resolve the case of discriminant $-115147$!  Watkins \cite{Watkins3} was able to extend Low's results to $-3\cdot 10^8<-D$ without exceptions.
\begin{figure}
\begin{center}
\includegraphics[scale=.8, viewport=0 0 350 400,clip]{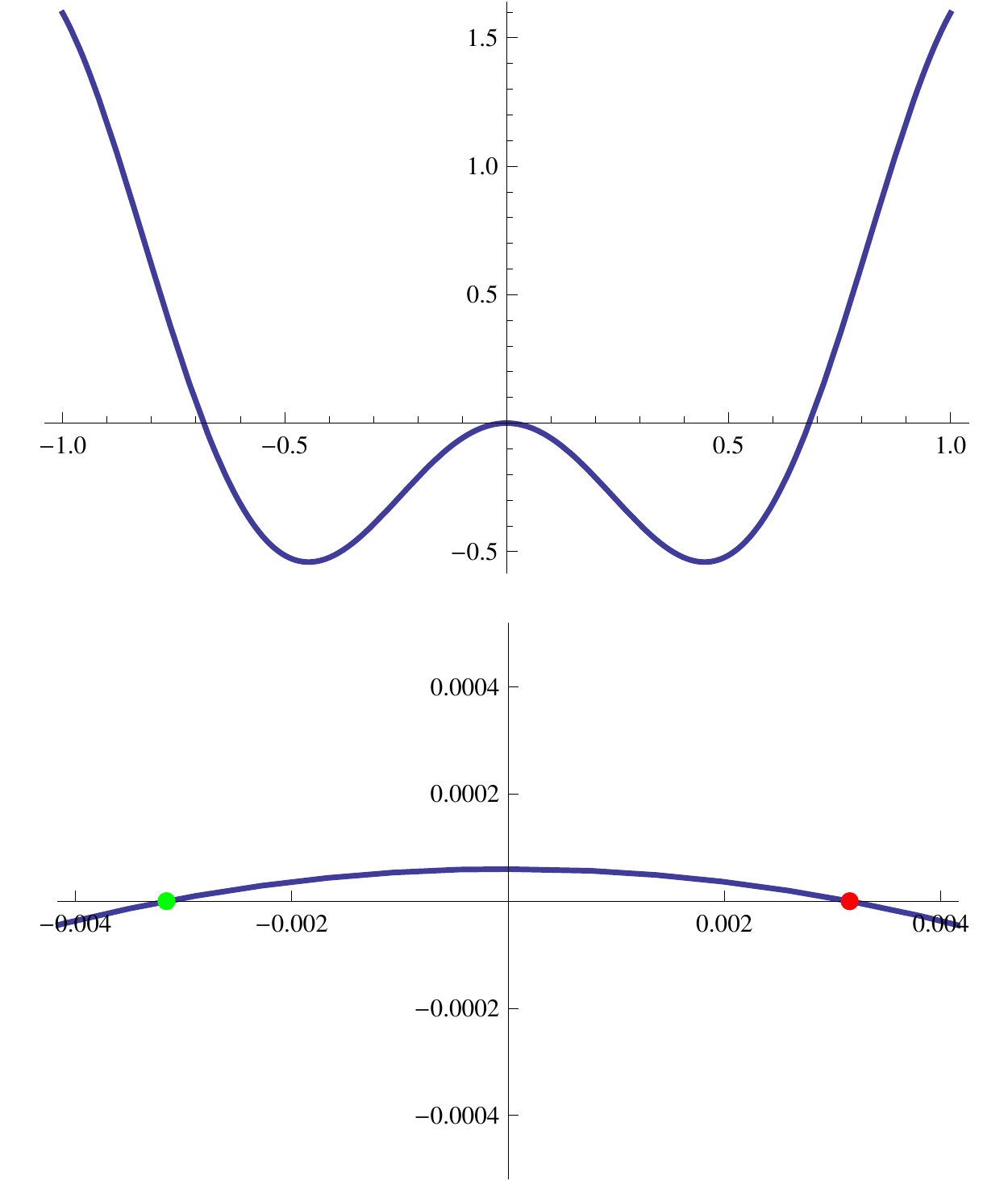}
\caption{$Z(t,\chi)$ for $-D=-115\, 147$.}
\label{F:1}
\end{center}
\end{figure}

The graph of $Z(t,\chi)$ for $-D=-115147$ is shown in Figure \ref{F:1}.  The value at $t=0.$ is $0.0000603627\ldots$\ \ The first zero \cite{W} is at $t=0.0031576\ldots$ , extremely small for a discriminant of this size.  (On average the lowest zero is at $1/\log(D/2\pi)$, which in this case works out to be $\approx0.1$)%This particular example of a low-lying zero has played a role in the solution of class number problems from \cite{MW} to \cite{Watkins}.  As far as I know, \cite{Low} is the first appearance of this example in the literature\footnote{Another \lq famous\rq\  example also first appears in \cite{Low}: $-D=-17\,923$ with low lying zero at $t=0.0309858\ldots$ .}

Since $Z(t,\chi)$ is an even function of $t$, $Z^\prime(0,\chi)=0$ and so
\[
\left(\log Z\right)^{\prime\prime}(0,\chi)=\frac{Z^{\prime\prime}(0,\chi)}{Z(0,\chi)}.
\]
We deduce from the Hadamard product for $\Xi(t,\chi)$ and \emph{Mathematica} evaluation of the derivative of the digamma function that
\begin{equation}\label{Eq:logzdoubleprime}
\left(\log Z\right)^{\prime\prime}(0,\chi)=-\sum_\gamma\frac{1}{\gamma^2}+0.635467\ldots ,
\end{equation}
where the sum is over the zeros $\gamma$ of $Z(t,\chi)$, or equivalently zeros $1/2+i\gamma$ of $L(s,\chi)$.
On $\GRHp$, the $\gamma$ are real and nonzero so $-\sum_\gamma 1/\gamma^2<0$.  Thus, 
\[
\left(\log Z\right)^{\prime\prime}(0,\chi)<0
\]
as soon as there are enough low-lying zeros for the sum to overcome the positive term $0.635467\ldots$ above.  %Since for discriminant $-D$, the lowest imaginary part of a zero $\gamma_1$ is asymptotic to $2\pi/\log(D)$, already we expect
%\begin{gather*}
%\left(\log Z\right)^{\prime\prime}(0,\chi)<0\quad \text{when}\\
%0.635467\ldots-(\log(D)/2\pi)^2<0
%\quad\text{or}\quad 150<D.
%\end{gather*}
One can easily compute, using Michael Rubinstein's package \texttt{lcalc}\footnote{see  \href{http://oto.math.uwaterloo.ca/~mrubinst/L_function_public/CODE/}{\texttt{http://oto.math.uwaterloo.ca/~mrubinst}}}, sufficient zeros for quadratic character $L$-functions with  $-10^4<-D$ to show that in fact
\[
\left(\log Z\right)^{\prime\prime}(0,\chi)<0\quad\text{for}\quad -10^4<-D<-119.
\]
On the other hand, from a theorem of Siegel  \cite[Theorem IV]{Siegel} it follow that there exists a universal constant $C_1$ such that for $D>C_1$, unconditionally
\begin{equation}\label{Eq:Siegel}
\gamma_1(-D)<\frac{4}{\log\log\log D},
\end{equation}
or
\[
\frac{1}{\gamma_1^2}>\frac{(\log\log\log D)^2}{16}.
\]
%this is the analog in conductor aspect of a result of Littlewood for $\zeta(s)$.  In \cite{Sami} it is shown that, on GRH, there is a universal constant $C_2$ such that
%\[
%\gamma_1<\frac{C_2}{\log\log D};
%\]

%From these results one can only deduce that (still on $\GRHp$), that
%\begin{gather*}
%\left(\log Z\right)^{\prime\prime}(0,\chi_{-D})<0\quad\text{for}\quad \max\left(C_1,\exp\left(3\cdot 10^{10}\right)\right)<D,\\
%\left(\log Z\right)^{\prime\prime}(0,\chi_{-D})<0\quad\text{for}\quad \exp\left(\exp\left(0.79716\cdot C_2\right)\right)<D
%\end{gather*}
%so there's not much point in trying to determine $C_1$ or $C_2$.

On $\GRHp$ we have that $Z(t,\chi)>0$ and, for $D$ sufficiently large, $Z^{\prime\prime}/Z(0,\chi)<0$ so $Z^{\prime\prime}(0,\chi)<0$.  Thus
\begin{proposition}
On $\GRHp$, for $D$ sufficiently large, $Z(0,\chi)$ is a positive local maximum.  We conjecture that $-D<-119$ is sufficient.
\end{proposition}
If GRH fails by reason of a Landau-Siegel zero, we would expect $Z(0,\chi)<0$, and, from (\ref{Eq:logzdoubleprime}), that $Z^{\prime\prime}/Z(0,\chi)>0$.  So  again $Z^{\prime\prime}(0,\chi)<0$, a negative local maximum.  Thus the example in Figure \ref{F:1} represents a near violation of GRH, and we will informally call such examples \lq Low discriminants\rq\  in analogy with Lehmer pairs for $\zeta(s)$.  A precise definition will be made in the last section below.  (There are nineteen fundamental discriminants $-119\le -D\le -3$ such that $Z(0,\chi)$ is a positive local \emph{minimum}.  This is simply analogous to the fact that the Hardy function $Z(t)$ does have a negative local maximum, at $t = 2.47575\ldots$)\ \ 

\subsection*{De Bruijn and Newman}

Since we are going to introduce the heat equation we have a clash of notations: $t=\tim(s)$ for $L(s,\chi)$ versus $t$ representing time in the heat equation.  So from now on we will write $\Xi(x,\chi)$ instead of $\Xi(t,\chi)$.  Following Polya \cite{Polya} and de Bruijn \cite{dB} we introduce a deformation parameter $t$:
\[
\Xi_t(x,\chi)=\int_0^\infty \exp(t u^2)\Phi(u,\chi)\cos(ux)\, du,
\]
so that for $t=0$, $\Xi_0(x,\chi)$ is just $\Xi(x,\chi)$.  This function satisfies the \textsc{backward heat equation}
\[
\frac{\partial \Xi}{\partial t}+\frac{\partial^2 \Xi}{\partial x^2}=0.
\]
We have that $\Xi_t(x, \chi)$ is an entire, even function.  Since $\Xi_0(x,\chi)$ is of order one, and
\[
\Xi_t(x)=\left(\sum_{m=1}^\infty\frac{(-1)^mt^m}{m!}\left(\frac{d}{dx}\right)^{2m}\right)\Xi_0(x),
\]
\cite[Theorem 11.4]{vanderSteen} gives that $\Xi_t(x,\chi)$ is of order at most one.
Since $\Phi(u,\chi)$ has  doubly exponential decay, \cite[Theorem 13]{dB} applies to $\Xi_t(x,\chi)$ and we have an analog of the theorem of de Bruijn for the Riemann zeta function:
\begin{enumerate}
\item For $t\ge 1/2$, $\Xi_t(x,\chi)$ has only real zeros.
\item If for some real $t$, $\Xi_t(x,\chi)$ has only real zeros, then $\Xi_{t^\prime}(x,\chi)$ also has only real zeros for any $t^\prime\ge t.$
\end{enumerate}
For this same reason, \cite[Theorem 3]{Newman} applies to $\Xi_t(x,\chi)$ and we have an analog of the theorem of Newman:  There exists a real constant $\Lambda_{-D}$, $-\infty<\Lambda_{-D}\le 1/2$, such that
\begin{enumerate}
\item $\Xi_t(x,\chi)$ has only real zeros if and only if $t\ge\Lambda_{-D}$.
\item $\Xi_t(x,\chi)$ has some complex zeros if $t<\Lambda_{-D}$.
\end{enumerate}\begin{definition}
We define
\[
\lkr=\sup\left\{\Lambda_{-D}\, |\, -D\text{ fundamental}\right\}.
\]
\end{definition}
\begin{conjecture}
Analogous to Newman's conjecture for the Riemann zeros, we conjecture that $\lkr\ge0$.
\end{conjecture}
\noindent
For the meaning of this we paraphrase by Newman's often quoted remark
\begin{quotation}
\lq\lq\emph{This new conjecture is a quantitative version of the dictum that the [Generalized] Riemann Hypothesis, if true, is only barely so.}\rq\rq
\end{quotation}
Under the GRH, $\lkr\le0$.  We are free to assume this, since its negation is $\lkr>0$ which implies the above conjecture.

In \cite{CSV} Csordas, Smith, and Varga give a precise, though somewhat technical, definition of a Lehmer pair of Riemann zeros.  Via the fact that the $t$-deformed Riemann xi function $\Xi_t(x)$ satisfies the backward heat equation, they were able to draw conclusions about the differential equation satisfied by the $k$-th gap between the zeros as the deformation parameter $t$ varies.  From this, they were able to use Lehmer pairs to give lower bounds on the de Bruijn-Newman constant $\Lambda$.  Our situation is exactly parallel, and we follow their exposition closely in the next section.

\subsection*{ODEs and the motion of the zeros}

\begin{lemma}[2.1 \cite{CSV}]  Suppose $x_0$ is a simple real zero of $\Xi_{t_0}(x,\chi)$.  Then in some open interval $I$ containing $t_0$, there is a real differentiable function $x(t)$ defined on $I$ and satisfying $x(t_0)=x_0$, such that $x(t)$ is a simple real zero of $\Xi_t(x,\chi)$ and $\Xi_t(x(t),\chi)\equiv0$ for $t\in I$.  Moreover, for $t\in I$,
\begin{equation}\label{Eq:ode}
x^\prime(t)=\frac{\Xi^{\prime\prime}_t(x(t),\chi)}{\Xi^\prime_t(x(t),\chi)}.
\end{equation}
NB: While on the left side $x^\prime$ obviously denotes derivative with respect to $t$, on the right side $\Xi^\prime$ denotes (confusingly) derivative with respect to $x$.
\end{lemma}
\begin{proof}
The existence of $x(t)$ follows directly from the Implicit Function Theorem.  Differentiate $\Xi_t(x(t))\equiv0$ with respect to $t$ to deduce
\begin{multline*}
0=\frac{d}{dt}\Xi_t(x(t),\chi)=\frac{\partial}{\partial t}\left.\Xi_t(x,\chi)\right|_{x=x(t)}+x^\prime(t)\frac{\partial}{\partial x}\left.\Xi_t(x,\chi)\right|_{x=x(t)}\\
=-\frac{\partial^2}{\partial x^2}\left.\Xi_t(x,\chi)\right|_{x=x(t)}+x^\prime(t)\frac{\partial}{\partial x}\left.\Xi_t(x,\chi)\right|_{x=x(t)}
\end{multline*}
by the backward heat equation.
\end{proof}

\begin{figure}
\begin{center}
\includegraphics[scale=.8, viewport=0 0 375 225,clip]{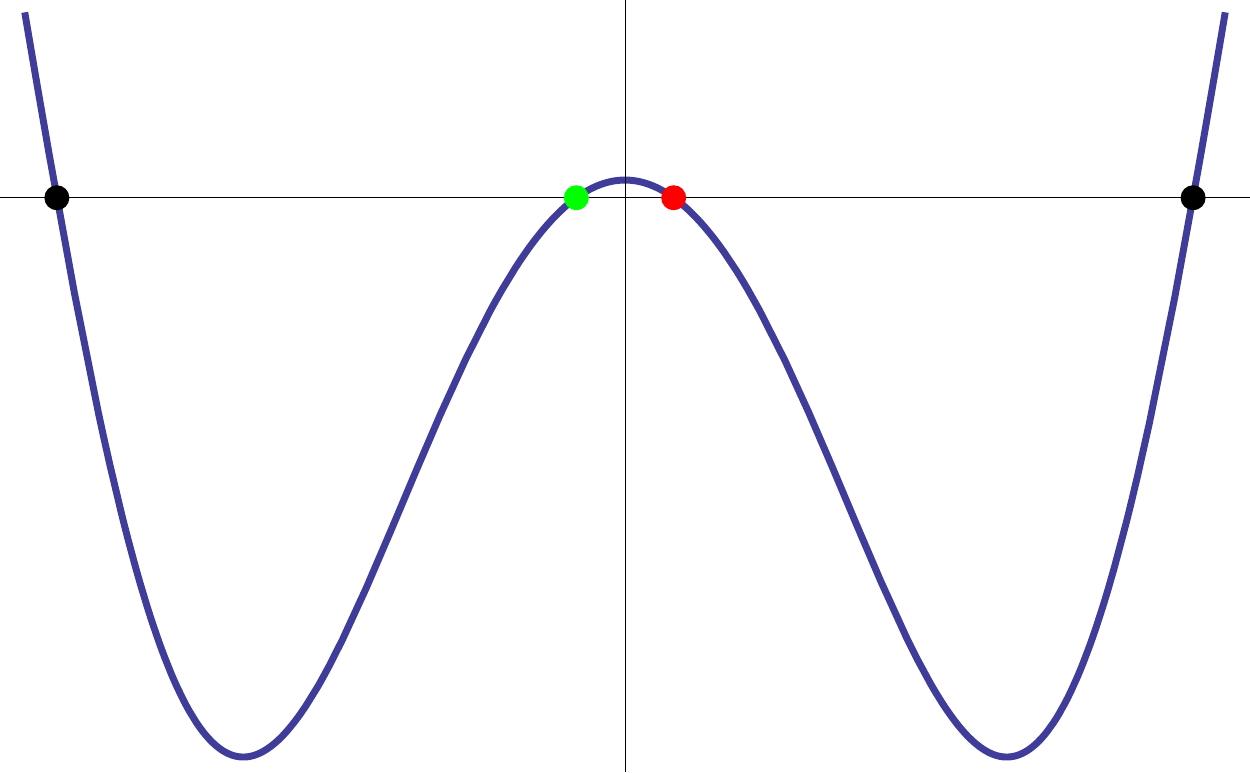}
\caption{Hypothetical sketch of $\Xi_{t_0}(x,\chi)$, with two close zeros marked in red and green.}\label{F:2}
\end{center}
\end{figure}

\begin{longquotation}[p.111-112,\cite{CSV}]\lq\lq The significance of Lemma 2.1 is that the movement of the simple real zero $x(t)$ of $\Xi_t(x,\chi)$ is locally determined solely by the ratio
\[
\frac{\Xi^{\prime\prime}_t(x(t),\chi)}{\Xi^\prime_t(x(t),\chi)}.
\]
To illustrate the result of Lemma 2.1, consider the graph of $\Xi_{t_0}(x,\chi)$ in Figure \ref{F:2}, where $\Xi_{t_0}(x,\chi)$ has two close\footnote{i.e., both in the same interval $I$ of the Lemma.  \emph{A priori} it is not obvious that such close zeros exist for any $D$!}  zeros $x_{-1}(t_0)$ and $x_{1}(t_0)$, and the remaining zeros are widely separated from $x_{-1}(t_0)$ and $x_{1}(t_0)$.  From the graph we see that
\[
\Xi^{\prime\prime}_t(x(t),\chi)<0
\]
on an interval containing $[x_{-1}(t_0), x_{1}(t_0)]$, and
\[
\Xi^{\prime}_{t_0}(x_{-1}(t_0),\chi)>0\quad\text{ while }\quad \Xi^{\prime}_{t_0}(x_{1}(t_0),\chi)<0.
\]
Using (\ref{Eq:ode}) we conclude from Figure \ref{F:2} that 
\begin{equation}
x^\prime_{-1}(t_0)<0\quad\text{ while }\quad x^\prime_{1}(t_0)>0,
\end{equation}
and this indicates that, on decreasing $t$, $x_{-1}(t)$ increases while $x_{1}(t)$ decreases, i.e., these two zeros move towards one another as $t$ decreases from $t_0$, and similarly, these two zeros move away from one another as $t$ increases from $t_0$.\rq\rq%\footnote{The situation is the same if the curve is inverted: negative local minimum between $x_{-1}(t_0)$ and $x_{1}(t_0)$.}
\end{longquotation}

What can we deduce if two roots coalesce?
\begin{lemma} Suppose for some real $t_0$ and $x_0$ we have a double root, i.e,
\[
\Xi_{t_0}(x_0,\chi)=0=\Xi^\prime_{t_0}(x_0,\chi)
\]
Then $t_0\le \Lambda_{-D}$.  For any $t$ with $\Lambda_{-D}<t$, the zeros are not only real but also simple.
\end{lemma}
\begin{proof}
Suppose first $\Xi^{\prime\prime}_{t_0}(x_0,\chi)\ne 0$.
We compute a Taylor expansion in two variables for the function $\Xi_{t_0\pm\delta^2}(x_0+\epsilon,\chi)$ out to second order terms, and use the backward heat equation to eliminate all partial derivatives with respect to $t$.  Then (up to an error term of third order in $\epsilon$ and $\delta^2$)
\begin{multline*}
\Xi_{t_0\pm\delta^2}(x_0+\epsilon,\chi)\approx \\
\Xi^{\prime\prime}_{t_0}(x_0,\chi)\cdot \left(\frac{\epsilon^2}{2}\mp\delta^2\right)
\mp\Xi^{\prime\prime\prime}_{t_0}(x_0,\chi)\cdot \delta^2\epsilon
+\Xi^{\prime\prime\prime\prime}_{t_0}(x_0,\chi)\cdot\frac{\delta^4}{2}.
\end{multline*}
The right side is a quadratic in $\epsilon$, with discriminant
\begin{equation}\label{Eq:discriminant}
\pm2\delta^2 \cdot \Xi^{\prime\prime}_{t_0}(x_0,\chi)^2+\delta^4\cdot \left(\Xi^{\prime\prime\prime}_{t_0}(x_0,\chi)^2-
\Xi^{\prime\prime}_{t_0}(x_0,\chi)\cdot \Xi^{\prime\prime\prime\prime}_{t_0}(x_0,\chi)\right)
\end{equation}

For $-\delta^2<0$, $\delta\ll 1$, the discriminant (\ref{Eq:discriminant}) is \emph{negative}, and so the quadratic in $\epsilon$ has no real roots but rather two complex conjugate roots.  Hence $t_0< \Lambda_{-D}$.

It is also interesting to note that for $+\delta^2>0$, $\delta \ll 1$, the discriminant is \emph{positive}, and so the quadratic has two simple real roots.  (In fact, for fixed $\delta\ll 1$ we can see explicitly the sign changes in $\Xi_{t_0+\delta^2}(x,\chi)$ for $x$ in the interval $[x_0-2\delta,x_0+2\delta]$.  Up to the error term of third order, our expansion looks like
\begin{align*}
\Xi_{t_0+\delta^2}(x_0-2\delta,\chi&)&&\approx&&\delta^2\cdot \Xi^{\prime\prime}_{t_0}(x_0,\chi)+2\delta^3\cdot \Xi^{\prime\prime\prime}_{t_0}(x_0,\chi)+\frac{\delta^4}{2}\cdot  \Xi^{\prime\prime\prime\prime}_{t_0}(x_0,\chi)\\
\Xi_{t_0+\delta^2}(x_0,\chi&)&&\approx&-&\delta^2\cdot \Xi^{\prime\prime}_{t_0}(x_0,\chi)+\frac{\delta^4}{2}\cdot  \Xi^{\prime\prime\prime\prime}_{t_0}(x_0,\chi)\\
\Xi_{t_0+\delta^2}(x_0-2\delta,\chi&)&&\approx&&\delta^2\cdot \Xi^{\prime\prime}_{t_0}(x_0,\chi)-2\delta^3\cdot \Xi^{\prime\prime\prime}_{t_0}(x_0,\chi)+\frac{\delta^4}{2}\cdot  \Xi^{\prime\prime\prime\prime}_{t_0}(x_0,\chi).
\end{align*}
The dominant $\delta^2$ term changes sign twice.)

If $\Xi^{\prime\prime}_{t_0}(x_0,\chi)=0$, a similar expansion to higher order gives the same result.
\end{proof}

\begin{remark}
Csordas \emph{et al.}\ \cite{CSV} give a different proof of the analogous result, although they remark one may give a proof using the backward heat equation, presumably similar to the above.  Surprisingly, this seems to be the first mention in the de Bruijn Newman constant literature, and the only mention in \cite{CSV}, of the backward heat equation.
\end{remark}

For $\Lambda_{-D}\le t$ and $k>0$, let $x_k(t)$ denote the $k$-th zero (by hypothesis real, simple, positive) of $\Xi_t$, so $x_k(0)=\gamma_k$. Given the symmetry of the zeros, we define $x_{-k}(t)=-x_k(t)$.  From the Hadamard factorization theorem (as a function of $x$ in $\mathbb C$) we have that
\begin{equation}\label{Eq:Hadamard}
\Xi_t(x)=\Xi_t(0)\prod_{k=1}^\infty\left(1-\frac{x^2}{x_k(t)^2}\right).
\end{equation}

We introduce the summation notation $\sum^\prime_j$ where the superscript $\prime$ denotes the sum omitting the (undefined) term with $j=0$, in addition to whatever side condition is additionally imposed on the summation variable.
\begin{lemma}[2.4 \cite{CSV}]  The zeros $x_k(t)$
are the solutions to the initial value problem
\begin{equation}\label{Eq:xode}
x_k^\prime(t)=\sum_{j\ne k}^\prime \frac{2}{x_k(t)-x_j(t)},\qquad x_k(0)=\gamma_k.
\end{equation}
\end{lemma}
\begin{proof}
Suppose $q(z)$ is some function analytic in a domain $D$, and $q(w)\ne0$.  Then for $f(z)=(z-w)q(z)$ we have that
\[
\frac{f^{\prime\prime}(w)}{f^\prime(w)}=2\frac{q^\prime(w)}{q(w)}.
\]
For example, fixing $k$ we have that
\[
\Xi_t(x,\chi)=(x-x_k(t))q_t(x),
\]
where
\[
q_t(x)=-\frac{\Xi_t(0,\chi)}{x_k(t)}\cdot\left(1+\frac{x}{x_k(t)}\right)\prod_{\substack{j=1\\j\ne k}}^\infty\left(1-\frac{x^2}{x_j(t)^2}\right).
\]
Then (writing $w=x_k(t)$)
\[
x_k^\prime(t)=\frac{\Xi_t^{\prime\prime}(w,\chi)}{\Xi_t^\prime(w,\chi)}=\frac{2q_t^\prime(w)}{q_t(w)}
=\frac{2}{w+x_k(t)}+\sum_{\substack{j=1\\j\ne k}}^\infty\frac{2}{w-x_j(t)}+\frac{2}{w+x_j(t)},
\]
and the lemma follows from $x_{-j}(t)=-x_j(t)$.
\end{proof}

\begin{figure}
\begin{center}
\includegraphics[scale=1, viewport=0 0 375 225,clip]{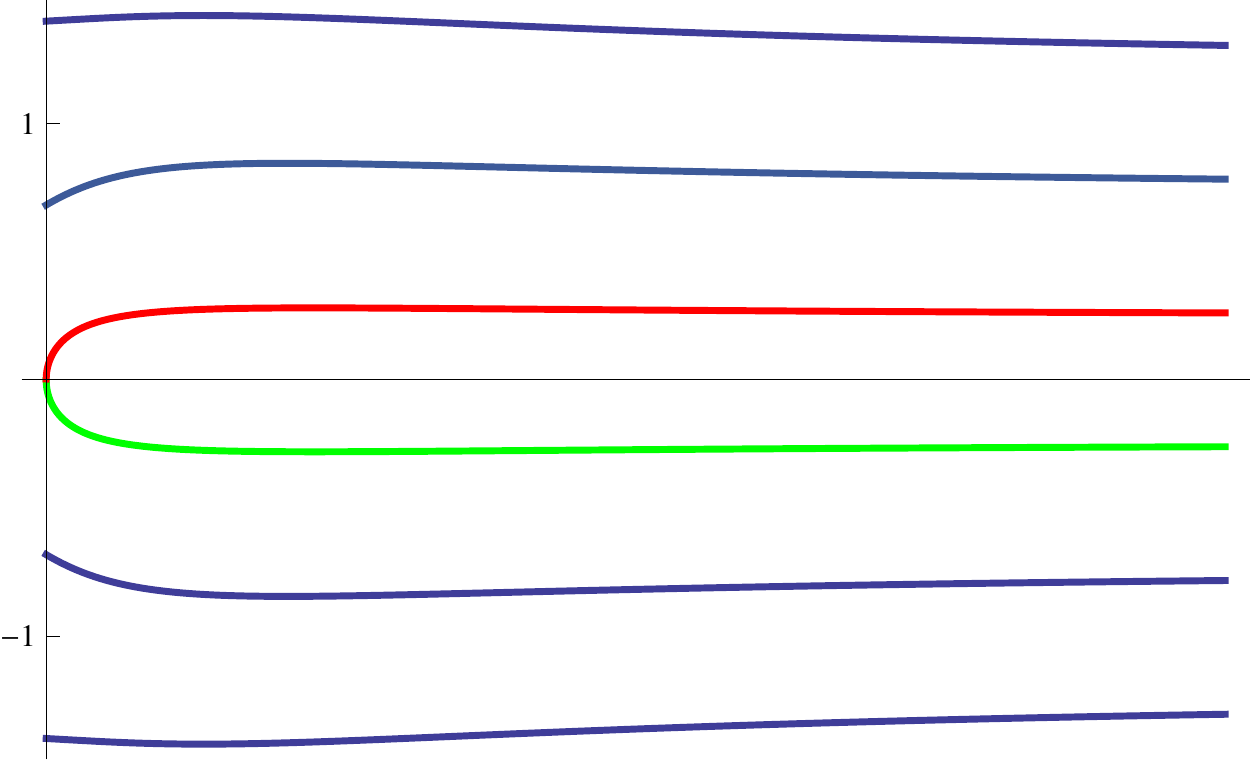}
\caption{Approximate numerical solution of (\ref{Eq:xode}) for $0\le t\le 1$.  Here $-D=-115\,147$, with trajectory of low lying zeros $x_{-1}(t),x_{1}(t)$ shown in green and red.}\label{F:3}
\end{center}
\end{figure}

\begin{lemma}[2.4 \cite{CSV}]  The first zero $x_1(t)$ satisfies the following differential equation:
\begin{equation}\label{Eq:ODE}
x_1(t)^\prime=\frac{1}{x_1(t)}-f(t)x_1(t),
\end{equation}
where
\begin{equation}\label{Eq:fsubk}
f(t)=\sum^\prime_{j\ne -1,1} \frac{2}{(x_{-1}(t)-x_j(t))(x_{1}(t)-x_j(t))}.
\end{equation}
\end{lemma}
\begin{proof}
From  (\ref{Eq:xode}) we have
\[
x_1^\prime(t)=\sum_{j\ne 1}^\prime \frac{2}{x_1(t)-x_j(t)}
,\qquad
x_{-1}^\prime(t)=\sum_{j\ne -1}^\prime \frac{2}{x_{-1}(t)-x_j(t)}.
\]
Subtract, and separate out the $j=-1$ term from the first sum and the $j=1$ term from the second to obtain
\begin{multline*}
x_1^\prime(t)-x_{-1}^\prime(t)=
\frac{2}{x_1(t)-x_{-1}(t)}-\frac{2}{x_{-1}(t)-x_1(t)}\\+
\sum^\prime_{j\ne -1,1}\left(
\frac{2}{x_1(t)-x_j(t)}-\frac{2}{x_{-1}(t)-x_j(t)}
\right).
\end{multline*}
The lemma follows from $x_1(t)-x_{-1}(t)=2x_1(t)$.
\end{proof}

\begin{lemma}  The initial value problem given by the ODE (\ref{Eq:ODE}) and $x_1(0)=\gamma_1$ has the solution
\begin{gather}
x_1(t)^2=\exp\left(-F(t)\right)\left(-2\int_t^0\exp\left(F(u)\right)\, dy+\gamma_1^2\right),\label{Eq:IVPsoln}\\
\text{where}\qquad F(t)=-2\int_t^0f(u)\, du.
\end{gather}
\end{lemma}
\begin{proof}
Multiply (\ref{Eq:ODE}) by $2x_1$ to get
\begin{gather*}
2x_1\cdot x_1^\prime=2-2f\cdot x_1^2\qquad\text{ or }\\
\frac{d}{dt}\left(x_1^2\right)+2fx_1^2=2.
\end{gather*}
An integrating factor is $\exp(F(t))$, where
\[
F(t)=2\int_0^t f(u)\, du.
\]
This gives
\begin{gather*}
\exp(F)\cdot \frac{d}{dt}\left(x_1^2\right)+2f\exp(F)x_1^2=2\exp(F),\\
\frac{d}{dt}\left(\exp(F)x_1^2\right)=2\exp(F),\\
\exp(F(t))x_1(t)^2=2\int_0^t\exp(F(u))\, du+x_1(0)^2.
\end{gather*}
We will be interested in $t<0$, so we swap the limits of integration and introduce the minus sign.
\end{proof}

\begin{remark}
Under the hypothesis that $\Lambda_{-D}\le t$, we have that $x_1(t)^2\ge 0$.  So any estimate for $f(t)$ which can show that
\[
2\int_{t_0}^0\exp\left(F(u)\right)\, du>\gamma_1^2,
\]
proves that $t_0< \Lambda_{-D}$.  This is the main idea of \cite{CSV}.  We now proceed to make such an estimate.
\end{remark}

Csordas \emph{et.\ al.}\ introduce a function analogous to
\begin{equation}\label{Eq:gsubk}
g(t)=\sum_{j\ne -1,1}^\prime\frac{1}{(x_{-1}(t)-x_j(t))^2}+\frac{1}{(x_{1}(t)-x_j(t))^2},
\end{equation}
\lq\lq in order to make the analysis of the movement of the zeros ... more tractable\rq\rq.  For $\Lambda_{-D}<t$ it is elementary that
\[
0<f(t)<g(t).
\]
They are able to show that
\begin{lemma}[2.5 \cite{CSV}] For $\Lambda_{-D}<t$
\[
g^\prime(t)>-8g(t)^2.
\]
\end{lemma}
This is a little technical so the proof is deferred.
From the lemma immediately follows
\begin{gather*}
\frac{1}{g(0)}-\frac{1}{g(t)}=-\int_t^0\frac{g^\prime(u)}{g(t)^2}\, du<8\int_t^0\, du=-8t,\\
g(t)<\frac{g(0)}{1+8g(0) t}\qquad\text{as long as}\qquad\frac{-1}{8g(0)}<t.
\end{gather*}
In turn, this gives
\begin{multline*}
F(t)=-2\int_t^0f(u)\, du> -2\int_t^0g(u)\, du>\\
-2\int_t^0\frac{g(0)}{1+8g(0) u}\, du=\frac{1}{4}\log\left(1+8g(0) t\right),
\end{multline*}
and
\begin{multline}\label{Ineq:key}
2\int_t^0\exp\left(F(u)\right)\, du
>2\int_t^0\left(1+8g(0) t\right)^{1/4}\, du\\
=\frac{\left(1-\left(1+8 g(0) t\right)^{5/4}\right)}{5g(0)}.
\end{multline}

%\begin{definition} We call a fundamental discriminant $-D<0$  a \textsc{Low discriminant}, if
%\begin{equation}\label{Eq:Lowdef}
%5\gamma_1^2 g(0)<1,
%\end{equation}
%where $L(1/2+i\gamma_1,\chi)=0$ and $g(t)$ is defined by (\ref{Eq:gsubk}).
%\end{definition}
\begin{theorem}\label{Theorem1}
Let $-D$ be any discriminant for which
\begin{equation}\label{Eq:Lowdef}
5\gamma_1^2 g(0)<1.
\end{equation}
Then by choosing in the inequality (\ref{Ineq:key})  the value of $t$ to be
\begin{equation}\label{Eq:lambdadef}
\lambda=\frac{(1-5\gamma_1^2 g(0))^{4/5}-1}{8g(0)},
\end{equation}
(note $-1/8g(0)<\lambda$) we have from \eqref{Ineq:key} that
\[
2\int_{\lambda}^0\exp\left(F(u)\right)du>\gamma_1^2.
\]
Then (\ref{Eq:IVPsoln}) shows that
\[
x_1^2(\lambda)<0\quad\text{ and so }\quad \lambda<\Lambda_{-D}\le \lkr.
\]
\end{theorem}
It will be helpful to have the series expansion
\begin{equation}\label{Eq:lambdaseries}
\lambda=-\frac{1}{2}\gamma_1^2\left(1+ \frac{\gamma_1^2 g(0)}{2}+O\left(\gamma_1^4g(0)^2\right)\right).
\end{equation}
\begin{remark}
Were we to continue to follow \cite{CSV}, the formal definition of Low discriminant would be any $-D$ for which (\ref{Eq:Lowdef}) holds.
It is reasonable to wonder, to what extent is such a definition natural, v. motivated by what we are able to prove?  How much do we give up by replacing $f(t)$ by $g(t)$?  How much do we give up when we use the bound $g^\prime(t)>-8g(t)^2$?  The interested reader will be able to verify, imitating what we did above, that any estimate of the form
\[
f^\prime(t)>-cf(t)^2,\qquad c>0
\]
would lead to a lower bound
\[
\lambda_c =\frac{(1-\frac{c+2}{c}\gamma_1^2 f(0))^{\frac{c}{c+2}}-1}{c\,f(0)}<0
\]
as long as
\[
\frac{c+2}{2}\gamma_1^2 f(0)<1.
\]
As above, it is useful to consider a series expansion for $\lambda_c$, to see the sensitivity to the various parameters.  We have
\[
\lambda_c=-\frac{1}{2}\gamma_1^2\left(1+ \frac{\gamma_1^2 f(0)}{2}+\left(\frac{1}{3}+\frac{c}{12}\right)\gamma_1^4 f(0)^2+O\left(c^2\gamma_1^6f(0)^3\right)\right)
\]
Thus we see that if we were able to prove a stronger theorem, we could make a definition that allowed more examples, but at the end of the day the bound we get from any example is still $\approx -1/2\, \gamma_1^2$.

Rather than give an \emph{ad hoc} definition which makes the theorem go through, we will postpone making a definition of Low discriminant until we have more insight.  Our definition will actually give \emph{fewer} examples, but (we hope) indicate why there might be infinitely many such.
\end{remark}

\subsubsection*{Proof of Lemma 2.5 \cite{CSV}}
We have (suppressing dependence on $t$)
\[
g^\prime(t)=-2\sum_{j\ne-1,1}^\prime\frac{x_{-1}^\prime-x_j^\prime}{(x_{-1}-x_j)^3}+\frac{x_{1}^\prime-x_j^\prime}{(x_{1}-x_j)^3}.
\]
Repeated applications of (\ref{Eq:xode}) show that we can write
\begin{multline*}
g^\prime(t)=-4\sum_{j\ne-1,1}^\prime\frac{1}{(x_{-1}-x_j)^3}\left(\sum_{i\ne-1}^\prime\frac{1}{x_{-1}-x_i}-\sum_{i\ne j}^\prime\frac{1}{x_j-x_i}\right)\\
-4\sum_{j\ne-1,1}^\prime\frac{1}{(x_{1}-x_j)^3}\left(\sum_{i\ne1}^\prime\frac{1}{x_{1}-x_i}-\sum_{i\ne j}^\prime\frac{1}{x_j-x_i}\right).
\end{multline*}
In the four inner sums, separate out the $i=j, i=-1, i=j, i=1$ terms to see that we can write 
\[
g^\prime(t)=A(t)+B(t),
\]
where
\[
A(t)=-8\sum_{j\ne-1,1}^\prime \frac{1}{(x_{-1}-x_j)^4}+\frac{1}{(x_{1}-x_j)^4}
\]
and
\begin{multline*}
B(t)=4\sum_{j\ne-1,1}^\prime\frac{1}{(x_{-1}-x_j)^2}\sum_{i\ne-1,j}\frac{1}{(x_{-1}-x_i)(x_j-x_i)}\\
+4\sum_{j\ne-1,1}^\prime\frac{1}{(x_{1}-x_j)^2}\sum_{i\ne 1,j}\frac{1}{(x_{1}-x_i)(x_j-x_i)}.
\end{multline*}
In $B(t)$ we separate out the $i=1$ term in the first double sum, and the $i=-1$ term in the second double sum to get that $B(t)=C(t)+D(t)$, where
\[
C(t)=4\sum_{j\ne-1,1}^\prime\frac{1}{(x_{-1}-x_j)^2(x_{1}-x_j)^2}
\]
and
\begin{multline*}
D(t)=4\sum_{j\ne-1,1}^\prime \sum_{i\ne-1,1,j}^\prime
\left\{
\frac{1}{(x_{-1}-x_j)^2(x_{-1}-x_i)(x_j-x_i)}\right.\\
+
\left.\frac{1}{(x_{1}-x_j)^2(x_{1}-x_i)(x_j-x_i)}
\right\}
\end{multline*}
We rewrite $D(t)$ as $D(t)/2+D(t)/2$, the sum of two (identical) double sums, and interchange the roles of $i$ and $j$ in the second:
\begin{multline*}
D(t)=2\sum_{j\ne-1,1}^\prime \sum_{i\ne-1,1,j}^\prime
\left\{
\frac{1}{(x_{-1}-x_j)^2(x_{-1}-x_i)(x_j-x_i)}\right.\\
+
\left.\frac{1}{(x_{1}-x_j)^2(x_{1}-x_i)(x_j-x_i)}
\right\}\\
+2\sum_{i\ne-1,1}^\prime \sum_{j\ne-1,1,i}^\prime
\left\{
\frac{1}{(x_{-1}-x_i)^2(x_{-1}-x_j)(x_i-x_j)}\right.\\
+
\left.\frac{1}{(x_{1}-x_i)^2(x_{1}-x_j)(x_i-x_j)}
\right\}.
\end{multline*}
Both double sums range over the same set of indices: all distinct $i, j$ taken from $\mathbb Z\backslash \{-1,0,1\}$.  So we may combine the first and third fraction over a common denominator, and also the second and fourth to get
\[
D(t)=2\sum_{j\ne-1,1}^\prime \sum_{i\ne-1,1,j}^\prime\frac{1}{(x_{-1}-x_j)^2(x_{-1}-x_i)^2}+\frac{1}{(x_{1}-x_j)^2(x_{1}-x_i)^2}.
\]
Both $C(t)$ and $D(t)$ are positive for $\Lambda_{-D}<t$, so $B(t)>0$ and
\[
g^\prime(t)>A(t)>-8g(t)^2.
\]
\qed

Following \cite{CSV} we can conclude
\begin{theorem}\label{Theorem2}
Suppose there exist infinitely many discriminants satisfying (\ref{Eq:Lowdef}).  Then $0\le\Lambda_{Kr}$.
\end{theorem}
\begin{proof} We have that for
\begin{gather*}
0<5\gamma_1(-D)^2 g(0,-D)\overset{\text{def.}}=u<1\\
\intertext{ we have}
\frac{\lambda(-D)}{\gamma_1(-D)^2}=\frac{5}{16}\cdot\frac{\left((1-u)^{4/5}-1\right)}{u}.
\end{gather*}
The function $f(u)$ on the right side above satisfies 
\[
-5/16<f(u)<-1/4\quad\text {for }\quad0<u<1.
\]
Since $\gamma_1(-D)^2\to 0$ via (\ref{Eq:Siegel}), then $\lambda(-D)\to 0$ as well.
\end{proof}

\begin{table}
\begin{center}
\begin{tabular}{ l c c c} 
$-D$ &          $\gamma_1$  & $\gamma_1\cdot \log(D/2\pi)$&$-\gamma_1^2/2$ \\ \hline\hline
 $-163$&            $0.202901$&$0.66062$&$-2.05844\cdot 10^{-2}$ \\
 $-1411$&          $0.077967$&$0.04221$&$-3.03943\cdot 10^{-3}$ \\
 $-17923$&        $0.030986$&$0.24652$&$-4.80057\cdot 10^{-4}$\\
 $-115147$&      $0.003158$&$0.03099$&$-4.98648\cdot 10^{-6}$ \\
 $-175990483$&$0.000475$&$0.00814$&$-1.12813\cdot 10^{-7}$ \\&&&
\end{tabular}
\caption{Examples of discriminants with low lying zeros.}\label{Ta:Low}
\end{center}
\end{table}

\subsection*{Numerical experiments}  Because of the applications to bounds on class numbers of positive definite binary quadratic forms \cite{MW, Watkins2},  examples of fundamental discriminants with low lying zeros are well studied; several are shown in Table \ref{Ta:Low}.  
%For  $D$ modest in size, these can be verified with M.\ Rubinstein's package \texttt{lcalc} \cite{Rubinstein}.  For larger $D$, there is the method \cite{Stopple1, Stopple2} of the PI, which is optimized in conductor aspect.  
Via (\ref{Eq:lambdaseries}) we expect $\lambda\approx -\gamma_1^2/2$ to be a lower bound.  Observe that from (\ref{Eq:lambdadef}) we have
\[
\frac{d\lambda}{d g(0)}=\frac{-1+g(0) \gamma_1 ^2+(1-5 g(0) \gamma_1 ^2)^{1/5}}{8 g(0)^2(1-5 g(0) \gamma_1 ^2)^{1/5}},
\]
and $-1 + y + (1 - 5 y)^{1/5}<0$ for $0<y<1/5$ implying that $\lambda$ is a decreasing function of $g(0)$.  Thus to get a lower bound for $\lambda$ it suffices to upper bound $g(0)$.  Furthermore, (\ref{Eq:lambdaseries}) indicates that the value of $\lambda$ is relatively insensitive to the tightness of this bound.

Via $\gamma_{-j}=-\gamma_j$ we determine that
\begin{align*}
g(0)=&2\sum_{j=2}^\infty\frac{1}{(\gamma_j+\gamma_1)^2}+\frac{1}{(\gamma_j-\gamma_1)^2}\\
=&2\sum_{j=2}^\infty\frac{1}{\gamma_j^2}\cdot \left(\frac{1}{(1+\gamma_1/\gamma_j)^2}+\frac{1}{(1-\gamma_1/\gamma_j)^2}\right).
\end{align*}
Let $N$ be such that $\gamma_N>1$.  Since 
\[
\frac{1}{(1+y)^2}+\frac{1}{(1-y)^2}=2\cdot \frac{1+y^2}{(1-y^2)^2}
\]
 is an increasing function on $(0,1)$, we can bound $g(0)$  by replacing $\gamma_1/\gamma_j$ by $\gamma_1$ for those $j\ge N$.  Thus 
 \begin{multline}\label{Eq:gbound}
 g(0)\le 2\sum_{2\le j<N}\frac{1}{(\gamma_j+\gamma_1)^2}+\frac{1}{(\gamma_j-\gamma_1)^2}\\
 +4\frac{1+\gamma_1^2}{(1-\gamma_1^2)^2}\,\sum_{N\le j}^\infty\gamma_j^{-2}.
 \end{multline}

From the Hadamard product
\[
\Xi(t,\chi)=\Xi(0,\chi)\prod_{j=1}^\infty\left(1-\frac{t^2}{\gamma_j^2}\right),
\]
we see that
\[
-\frac{1}{2}\frac{\Xi^{\prime\prime}(0,\chi)}{\Xi(0,\chi)}=\sum_{j=1}^\infty \gamma_j^{-2}.
\]
So to bound $g(0)$ it suffices to compute only the first $N$ zeros, and then numerically integrate the moments
\begin{equation}\label{Eq:momentdef}
\Xi(0,\chi)=\int_0^\infty \Phi(u,\chi) \, du,\qquad \Xi^{\prime\prime}(0,\chi)=-\int_0^\infty u^2\Phi(u,\chi) \, du,
\end{equation}
where recall that $\Phi(u,\chi)$ is defined by (\ref{Eq:phidef}).  Now (\ref{Eq:gbound}) becomes
\[
g(0)\le g(0)_{\text{bound}},
\]
where
 \begin{multline}\label{Eq:gbound1}
 g(0)_{\text{bound}}= 2\sum_{2\le j<N}\frac{1}{(\gamma_j+\gamma_1)^2}+\frac{1}{(\gamma_j-\gamma_1)^2}\\
 -\frac{1+\gamma_1^2}{(1-\gamma_1^2)^2}\,\left(2\frac{\Xi^{\prime\prime}(0,\chi)}{\Xi(0,\chi)}+4\sum_{1\le j<N}\gamma_j^{-2}\right).
 \end{multline}

It is easy to compute the moments \eqref{Eq:momentdef} in \emph{Mathematica}; we need only convince the reader we can bound the truncation error in the improper integral and infinite series.  We have
\[
\left|\Phi(u,\chi)\right|<\int_1^\infty x\exp(3/2 u-x^2\exp(2u)/D)\, dx<D\exp(-\exp(2u)/D).
\]
Thus we can bound the tails of the integrals
\begin{multline*}
\int_U^\infty \Phi(u,\chi)\, du<\int_U^\infty u^2 \Phi(u,\chi)\, du<\int_U^\infty Du^2 \exp(-\exp(2u)/D) \, du\\
<\int_U^\infty D\exp(2u) \exp(-\exp(2u)/D) \, du<D\exp(-\exp(2U)/D).
\end{multline*}
We see that for $U=\log(D\log(2\cdot 10^{15}D^2))$, the truncation error in the improper integral is less than $5\cdot 10^{-16}$.  Next we desire to bound the tail of the infinite series in order to compute $\Phi(u,\chi)$ with an error of less than $5\cdot 10^{-16}/U$, in order that the accumulated error in the integral over $[0,U]$  is less than $5\cdot 10^{-16}$.  Again we estimate
\begin{multline*}
\left|\sum_{n=N}^\infty \chi(n)n\exp(3u/2-n^2\pi\exp(2u)/D)\right|\\
<\int_{x=N}^\infty x\exp(3/2 u-x^2\exp(2u)/D)\, dx\\
<D \exp(-N^2 \exp2 u)/D).
\end{multline*}
Thus we need
\begin{align*}
N(u)=&D^{1/2}\exp(-u)\log(2\cdot 10^{15}D U )^{1/2}\\
=&D^{1/2}\exp(-u)\log(2\cdot 10^{15}D \log(D\log(2\cdot 10^{15}D^2)) )^{1/2}
\end{align*}
terms of the series, as a function of the variable $u$.  The error in computing the moments is less than $10^{-15}$.

\begin{table}
\begin{center}
\begin{tabular}{ l c c } 
$-D$ &          $\gamma_1$  &$\lambda$ \\ \hline\hline
 $-163$&            $0.202901$&$-2.15787\cdot 10^{-2}$ \\
 $-1411$&          $0.077967$&$-3.07533\cdot 10^{-3}$ \\
  $-17923$&        $0.030986$&$-4.81901\cdot 10^{-4}$\\
 $-115147$&      $0.003158$&$-4.98563\cdot 10^{-6}$ \\
 %$-124414899$&$0.000864$&$0.01451$\\
 $-175990483$&$0.000475$&$-1.12929\cdot 10^{-7}$ \\&&
\end{tabular}
\caption{Examples of discriminants with low lying zeros, and corresponding bound on $\lkr$.}\label{Ta:Low2}
\end{center}
\end{table}

We compute in Table \ref{Ta:Low2} that each of the discriminants in Table \ref{Ta:Low} satisfies (\ref{Eq:Lowdef}), and give the corresponding lower bound of Theorem \ref{Theorem1} on $\lkr$.  These computations were verified  in several ways:
\begin{enumerate}
\item Values of Dirichlet $L$-functions are independently implemented in \emph{Mathematica}.
The zero moment 
$
\Xi(0,\chi)
$ was compared to
\[
\quad(D/\pi)^{3/4}\Gamma(3/4)L(1/2,\chi),
\]
giving the same values (to 25 digits).
\item The first $10^4$ zeros were computed with \texttt{lcalc}.  The ratio 
\[
-\frac{1}{2}\frac{\Xi^{\prime\prime}(0,\chi)}{\Xi(0,\chi)}
\quad\text{ was compared to }\quad
\sum_{j=1}^{10^4} \gamma_j^{-2},
\]
with good accuracy.
\item  The upper bound $g(0)_{\text{bound}}$ was compared to a numerical estimate of $g(0)$ using the same first $10^4$ zeros, and achieved the desired inequality.
\item Finally, in all cases we observe that $\lambda\approx -\gamma_1/2$, as predicted.
\end{enumerate}

The example of $-D=-175990483$ required extra care, in that $g(0)_{\text{bound}}$ is the difference of two very large but approximately equal numbers, leading to potentially significant cancellation error.  The package \texttt{lcalc}, even compiled with double precision, did not compute zeros to sufficient accuracy.  Instead the method of \cite{Stopple1, Stopple2} was used to compute the zeros with $\gamma<1$ to $25$ digits of accuracy.  In this example, we obtain that 
\[
5\gamma_1^2\cdot g(0)_{\text{bound}}=0.00008,
\]
 sufficiently less than $1$ that we are confident of the results.

\begin{theorem}  
We have that $-D=-175990483$ satisfies (\ref{Eq:Lowdef}), and the corresponding zero gives the bound 
\[
-1.12929\cdot 10^{-7}<\lkr.
\]
\end{theorem}

\subsection*{Random matrix theory}
In \cite{Odlyzko, Odlyzko1} Odlyzko presents heuristic arguments that random matrix theory predictions for the Riemann zeros lead one to believe Newman's conjecture $\Lambda\ge 0$.  It seems to be difficult to make these more than heuristic.  The present case of quadratic Dirichlet $L$-functions appears to be different, because the functional equation for $Z(t,\chi)$ transposes our close pair of zeros $\gamma_1$ and $-\gamma_1$.  This distinction is largely the motivation for the present paper.

The standard conjectures \cite{KS} predict that low lying zeros of quadratic Dirichlet $L$-functions should be distributed according to a symplectic random matrix model.  In particular, \cite{KS} shows that there exists a probability measure $\nu(-,j)$ such that
\[
\lim_{N\to\infty}\nu_j(USp(2N))=\nu(-,j),
\]
where $\nu_j(USp(2N))$ gives the distribution of the $j$-th eigenvalue of a random matrix from $USp(2N)$.  If we normalize the zeros via $\widetilde \gamma_j=\gamma_j\log(D/2\pi)$, then as $D$ varies the $\widetilde \gamma_j$ are predicted be distributed according to $\nu(-,j)$.  

We can now re-write
\[
g(0)=2\sum_{j=2}^\infty\gamma_j^{-2}\cdot \left(\frac{1}{(1+\widetilde\gamma_1/\widetilde\gamma_j)^2}+\frac{1}{(1-\widetilde\gamma_1/\widetilde\gamma_j)^2}\right)
\]
We suppose an extra condition on the discriminants: that $\widetilde\gamma_j\ge 1$ for $j\ge 2$.  A positive proportion\footnote{In fact, for most discriminants since the mean of $\nu(-,2)$ is about $1.76\ldots$} of discriminants are predicted to satisfy this. 
(This holds for all discriminants in Table \ref{Ta:Low}.)\ \   
  Under this hypothesis we can bound $g(0)$ by replacing $\widetilde\gamma_1/\widetilde\gamma_j$ by $\widetilde\gamma_1$.  Thus 
 \begin{equation}\label{Eq:gboundtilde}
 g(0)\le 4\frac{1+\widetilde\gamma_1^2}{(1-\widetilde\gamma_1^2)^2}\,\sum_{j=2}^\infty\gamma_j^{-2}.
 \end{equation}

We can use the fact that for $y\ge 0$
\[
1-3y^2\le \frac{(1-y^2)^2}{1+y^2},
\]
and rearrange the terms in (\ref{Eq:Lowdef}) and (\ref{Eq:gboundtilde}) to see that a sufficient condition for (\ref{Eq:Lowdef}) is that $\widetilde\gamma_2(-D)\ge 1$ and
\[
-\frac{1}{2} \cdot \frac{\Xi^{\prime\prime}(0,\chi)}{\Xi(0,\chi)}  <\frac{21}{20}\cdot \gamma_1(-D)^{-2}-\frac{3}{20}\cdot \log(D/2\pi)^2.
\]
(For context, observe the expression on the left, as a series, begins with $\gamma_1(-D)^{-2}$.)  

Now $\Xi^{\prime\prime}/\Xi(0,\chi)$ differs from $\log(L(1/2,\chi))^{\prime\prime}$ by the derivative of the digamma function at $3/4$:
\[
-\frac{1}{4}\psi^\prime(3/4)\approx -0.63547\ldots,\quad\text{where}\quad \psi(z)=\frac{\Gamma^\prime(z)}{\Gamma(z)}.
\]
Furthermore we have that for $D>100$,
\[
\frac{1}{20}\left(\log(D/2\pi)\right)^2>\frac{1}{8}\psi^\prime(3/4).
\]
\begin{definition} We call a fundamental discriminant $-D<0$  a \textsc{Low discriminant}, if $\gamma_2(-D)\log(D/2\pi)\ge 1$ and
\begin{equation}\label{Eq:Low3}
-\frac{1}{2} \cdot \log\left(L(1/2,\chi)\right)^{\prime\prime}  <\frac{21}{20}\cdot \gamma_1(-D)^{-2}
-\frac{1}{5}\cdot \log(D/2\pi)^2.
\end{equation}
This is a sufficient condition for (\ref{Eq:Lowdef}) to hold.  Table \ref{Ta:Low3} shows examples; note this criterion \emph{fails} for $-D=-163$.
\end{definition}

This definition is motivated by the random matrix theory.  In \cite{KeSn}, the authors use the characteristic polynomial of a random matrix from $USp(2N)$, with $2N\approx \log(D/2\pi)$ to model $L(1/2+i t,\chi)$.  One might hope to show that a random matrix analog of (\ref{Eq:Low3}) holds with a positive probability.
This would give, under random matrix theory predictions for the distributions of the zeros, an infinite sequence of Low discriminants with associated
$
\lambda(-D)
$
lower bound for $\Lambda_{Kr}$.  By Theorem \ref{Theorem2}, random matrix theory predictions for the distribution of the zeros would imply the Generalized Newman Conjecture.

\begin{table}
\begin{center}
\begin{tabular}{ l c c } 
$-D$ &          $-0.5\log\left(L(1/2,\chi)\right)^{\prime\prime}$  &$1.05 \gamma_1^{-2}
-0.2\log(D/2\pi)^2$ \\ \hline\hline
 $-163$&            $25.0367$&$23.3845$ \\
 $-1411$&          $165.731$&$166.867$ \\
  $-17923$&        $1043.82$&$1080.95$\\
 $-115147$&      $100299.$&$105291.$ \\
 %$-124414899$&$0.000864$&$0.01451$\\
 $-175990483$&$4.4276\cdot 10^6$&$4.6489\cdot 10^6$ \\&&
\end{tabular}
\caption{Examples of Low discriminants.}\label{Ta:Low3}
\end{center}
\end{table}

\subsubsection*{Acknowledgements}  Thanks to Steven J. Miller for his careful reading of the manuscript and helpful suggestions.


\begin{thebibliography}{99}
\bibitem{dB} N.G. de Bruijn, \emph{The roots of trigonometric integrals}, Duke J. Math, \textbf{17} (1950), pp. 197-226.
%\bibitem{CVN2} G. Csordas, T.\ Norfolk, R.\ Varga, \emph{A lower bound for the de Bruijn-Newman constant}, Numer. Math. \textbf{52} (1988), pp.483-497.
\bibitem{CSV} G. Csordas, W. Smith, R. Varga, \emph{Lehmer pairs of zeros, the de Bruijn-Newman constant, and the Riemann Hypothesis}, Constructive Approximation, \textbf{10} (1994), pp. 107-129.
\bibitem{COSV}  G.\ Csordas, A.\ Odlyzko, W.\  Smith, R.\ Varga, \emph{A new Lehmer pair of zeros and a new lower bound for the de Bruijn-Newman constant $\Lambda$}, Electron.\ Trans.\ Numer.\ Anal., \textbf{1} (1993), pp.\ 104-111.
%\bibitem{Edwards} H.M. Edwards, Riemann's Zeta Function, Dover, 2001.
%\bibitem{Ivic} A. Ivi\'c, \emph{On some reasons for doubting the Riemann hypothesis}, preprint.
\bibitem{KS} N. Katz and P. Sarnak, Random Matrices, Frobenius Eigenvalues, and Monodromy, AMS Colloquium Publications \textbf{45}, 1999.
\bibitem{KeSn} J.\ Keating and N.\ Snaith, \emph{Random matrix theory and $L$-functions at $s=1/2$},  Comm. Math. Phys. \textbf{214} (2000), pp.\  91-110.
%\bibitem{Lehmer} D.H. Lehmer, \emph{On the roots of the Riemann Zeta Function}, Acta Math. \textbf{95} (1956), pp.291-298.
%\bibitem{XjL} X.J.\ Li, \emph{A formula for the Dedekind $\xi$-function of an imaginary quadratic field}, Journal of Mathematical Analysis and Applications \textbf{260} (2001), pp. 404-420.
\bibitem{Low} M.E. Low, \emph{Real zeros of the Dedekind zeta function of an imaginary quadratic field}, Ph.D. thesis, University of Illinois Urbana-Champaign (1965)
%\bibitem{Low2} \bysame, \emph{Real zeros of the Dedekind zeta function of an imaginary quadratic field} Acta Arith \textbf{14} (1967/1968), pp. 117-140.
\bibitem{MW} H. Montgomery, P. Weinberger, \emph{Notes on small class numbers}, Acta Arith. \textbf{XXIV} (1974), pp. 529-542.
\bibitem{Newman} C.M. Newman, \emph{Fourier transforms with only real zeros}, Proc. AMS, \textbf{61} (1976), pp. 245-251.
\bibitem{Odlyzko}   A.M. Odlyzko, \emph{The $10^{20}$-th zero of the Riemann zeta function and $175$ million of its neighbors}, \href{http://www.dtc.umn.edu/~odlyzko/unpublished/zeta.10to20.1992.pdf}{preprint}.
\bibitem{Odlyzko1} A.M. Odlyzko, \emph{An improved bound for the de Bruijn-Newman constant}, Numerical Algorithms \textbf{25} (2000), pp. 293-303.
\bibitem{Polya} G.\ Polya, \emph{\"Uber trigonometrische Integrale mit nur reelen Nullstellen}, J.\ f\"ur die reine und angewandte Mathematik \textbf{158} (1927), pp. 6-18.
%\bibitem{Rubinstein}  M.\ Rubinstein, \href{http://oto.math.uwaterloo.ca/~mrubinst/L_function_public/CODE/}{\texttt{lcalc}}
%\bibitem{Sami} O.\ Sami, \emph{Non-vanishing of Dirichlet $L$-functions at the central point.} in Algorithmic Number Theory,  Lecture Notes in Computer Science \textbf{5011}, Springer 2008, pp. 443Ð453.
%\bibitem{Saouter}Y.\ Saouter, X.\ Gourdon, P.\ Demichel,  \emph{An improved lower bound for the de Bruijn-Newman constant}, Math.\ Comp.\ \textbf{80} (2011), pp.\ 2281-2287.
\bibitem{Siegel} C.L. Siegel, \emph{On the zeros of Dirichlet $L$-functions}, Annals of Mathematics, \textbf{46} (1945) no. 3, pp. 409-422.
\bibitem{Stopple1} J. Stopple, \href{http://www.ams.org/mcom/2007-76-260/S0025-5718-07-01994-1/S0025-5718-07-01994-1.pdf}{\emph{Computing $L$-functions with large conductor}}, Mathematics of Computation, \textbf{76} (2007), pp.\ 2051-2062.
\bibitem{Stopple2} \bysame \href{http://akpeters.metapress.com/content/r125645t7g6375gu}{\emph{The quadratic character experiment}}, Experimental Mathematics, \textbf{18} (2009), pp.\ 193-200.
%\bibitem{Tit}E. Titchmarsh, The Theory of the Riemann Zeta Function, Oxford University Press, 2nd ed., 1986.
\bibitem{vanderSteen} P.\ van der Steen, \emph{On differential operators of infinite order}, Ph.D. thesis, TU Delft, 1968.
%\bibitem{Watkins} M. Watkins, \emph{Class numbers of imaginary quadratic fields}, Ph.D. thesis, University of Georgia, 2000.
\bibitem{Watkins2}M. Watkins, \emph{Class numbers of imaginary quadratic fields}, Math. Comp.,  \textbf{73} (2003), pp.\ 907-938.
\bibitem{Watkins3} \bysame, \emph{Real zeros of real odd Dirichlet $L$-functions}, Math. Comp., \textbf{73} (2003), pp.\ 415-423.
\bibitem{W} P. Weinberger, \emph{On small zeros of Dirichlet L-functions}, Math. Comp., \textbf{29} (1975), pp.\ 319-328.
\end{thebibliography}
\end{document}